\documentclass[pdflatex,sn-mathphys-num]{sn-jnl}

\usepackage{graphicx}%
\usepackage{multirow}%
\usepackage{amsmath,amssymb,amsfonts}%
\usepackage{amsthm}%
\usepackage{mathrsfs}%
\usepackage[title]{appendix}%
\usepackage{xcolor}%
\usepackage{textcomp}%
\usepackage{manyfoot}%
\usepackage{booktabs}%
\usepackage{algorithm}%
\usepackage{algorithmicx}%
\usepackage{algpseudocode}%
\usepackage{listings}%
\usepackage{mathtools}%
\usepackage{diagbox}%

\theoremstyle{thmstyleone}%
\newtheorem{theorem}{Theorem}
\theoremstyle{thmstyletwo}%

\theoremstyle{thmstylethree}%
\newtheorem{definition}{Definition}%

\newcommand{\mbf}[1]{\mathbf{#1}}
\newcommand{\mbb}[1]{\mathbb{#1}}
\newcommand{\mcal}[1]{\mathcal{#1}}

\newcommand{\clr}{\mathrm{clr}}

\newcommand{\ind}{\mathrm{ind}}
\newcommand{\inte}{\mathrm{int}}
\newcommand{\sym}{\mathrm{sym}}
\newcommand{\skw}{\mathrm{skew}}
\newcommand{\syind}{\mathrm{syind}}
\newcommand{\skind}{\mathrm{skind}}
\newcommand{\syint}{\mathrm{syint}}
\newcommand{\skint}{\mathrm{skint}}
\newcommand{\qs}{\mathrm{QS}}
\newcommand{\gmh}{\mathrm{GMH}}

\providecommand\given{}
\newcommand\SetSymbol[1][]{%
  \nonscript\:#1\vert
  \allowbreak
  \nonscript\:
  \mathopen{}}
  \DeclarePairedDelimiterX\Set[1]\{\}{%
  \renewcommand\given{\SetSymbol[\delimsize]}
  #1
}

\begin{document}

\title[Quasi-symmetry and geometric marginal homogeneity: A simplicial approach to square contingency tables]{Quasi-symmetry and geometric marginal homogeneity: A simplicial approach to square contingency tables}


\author*[1]{\fnm{Keita} \sur{Nakamura}}\email{6324703@ed.tus.ac.jp}

\author[2,3]{\fnm{Tomoyuki} \sur{Nakamgawa}}\email{tomoyuki.nakagawa@meisei-u.ac.jp}

\author[1]{\fnm{Kouji} \sur{Tahata}}\email{kouji\_tahata@rs.tus.ac.jp}

\affil*[1]{\orgdiv{Department of Information Sciences}, \orgname{Tokyo University of Science}, \orgaddress{\city{Noda City}, \state{Chiba}, \postcode{278-8510}, \country{Japan}}}

\affil[2]{\orgdiv{School of Data Science}, \orgname{Meisei University}, \orgaddress{\city{Hino City}, \state{Tokyo}, \postcode{191-8506}, \country{Japan}}}

\affil[3]{\orgdiv{Statistical Mathematics Unit}, \orgname{RIKEN Center for Brain Science}, \orgaddress{\city{Wako City}, \state{Saitama}, \postcode{351-0198}, \country{Japan}}}


\abstract{
Square contingency tables are traditionally analyzed with a focus on the symmetric structure of the corresponding probability tables.
We view probability tables as elements of a simplex equipped with the Aitchison geometry.
This perspective allows us to present a novel approach to analyzing symmetric structure using a compositionally coherent framework.
We present a geometric interpretation of quasi-symmetry as an e-flat subspace and introduce a new concept called geometric marginal homogeneity, which is also characterized as an e-flat structure.
We prove that both quasi-symmetric tables and geometric marginal homogeneous tables form subspaces in the simplex, and demonstrate that the measure of skew-symmetry in Aitchison geometry can be orthogonally decomposed into measures of departure from quasi-symmetry and geometric marginal homogeneity.
We illustrate the application and effectiveness of our proposed methodology using data on unaided distance vision from a sample of women.
}




\maketitle

\section{Introduction}\label{sec1}
A contingency table (CT) is a tabular summary of categorical data, where each cell represents the frequency of observations falling into a specific combination of row and column categories.
When the row and column variables have the same $I$ categories, an $I \times I$ square CT is formed with the $(i, j)$-th cell, denoted by $n_{ij}$, representing the frequency of observations falling into the $i$-th row and $j$-th column categories.
The corresponding probability table (PT) consists of elements $p_{ij}$, which represents the probability of an observation falling into the $(i, j)$-th cell, with $\sum_{i=1}^{I}\sum_{j=1}^{I}p_{ij}=1$.

One of the most important aspects of the analysis of square CTs is the study of symmetry in the corresponding PTs, because it indicates the exchangeability between the row and column variables.
Symmetry requires the equality of the probabilities of symmetric cells, i.e.,
\[
  p_{ij}=p_{ji}\quad(i=1,\ldots,I;j=1,\ldots,I).
\]
\citet{bowker_test_1948} first formalized the concept of symmetry in square CTs by proposing a statistical test for the symmetry of PTs.

The symmetry assumption is often too limiting for practical data. To address this limitation, various less restrictive structures have been suggested, including marginal homogeneity \citep{stuart_test_1955} and quasi-symmetry \citep{caussinus_contribution_1965}.
Marginal homogeneity only requires the equality of the marginal distributions of rows and columns, that is,
\[
  p_{i+}=p_{+i}\quad(i=1,\ldots,I),
\]
where $p_{i+}=\sum_{j=1}^{I}p_{ij}$ and $p_{+i}=\sum_{j=1}^{I}p_{ji}$ are the marginal probabilities of the $i$th row and column, respectively.
Quasi-symmetry represents the symmetry of association in the table, as reflected in the symmetry of local odds ratios.
Let $\theta_{ij}$ denote the local odds ratio for the $2\times2$ subtable formed by the $i$-th and $(i+1)$-th rows and the $j$-th and $(j+1)$-th columns:
\[
  \theta_{ij} = \frac{p_{ij}p_{i+1,j+1}}{p_{i,j+1}p_{i+1,j}}\quad(i=1,\ldots,I-1;j=1,\ldots,I-1).
\]
Quasi-symmetry requires that these local odds ratios are symmetric, that is,
\begin{equation}
  \label{eq1}
  \theta_{ij} = \theta_{ji}\quad(i=1,\ldots,I-1;j=1,\ldots,I-1).
\end{equation}
This condition allows for different marginals while maintaining symmetry in the associational structure of the table.

\citet{caussinus_contribution_1965} showed that a square PT satisfies symmetry if and only if it satisfies both marginal homogeneity and quasi-symmetry.
These structures can be understood from an information-geometric viewpoint.
The symmetric structure forms not only an e-flat but also an m-flat submanifold of the probability simplex equipped with the Fisher metric.
The quasi-symmetric structure forms an e-flat submanifold, while the marginal homogeneous structure forms an m-flat submanifold.
These two submanifolds create a foliation of the probability simplex, with their intersection precisely corresponding to the symmetric structure.
These structures provide a foundation for more flexible structure of square CTs, capturing important aspects of symmetry while allowing for deviations that often occur in real-world data.

Over the past few decades, the field of compositional data analysis has seen significant advances, particularly with the introduction of an algebraic-geometric structure known as Aitchison geometry on the simplex \citep{aitchison_statistical_1982, 10.5555/17272, pawlowsky-glahn_geometric_2001}. 
This geometric framework provides a coherent approach to analyzing data that represent parts of a whole and are constrained to sum to a constant, such as $1$ or $100\%$. 
By treating PTs as compositional data, we can apply these analytical tools to the study of square CTs.
\citet{egozcue_compositional_2008, egozcue_independence_2015} pioneered this approach by examining independence and interaction between row and column variables in CTs using Aitchison geometry.
\citet{facevicova_logratio_2014, facevicova_compositional_2016, facevicova_general_2018} developed and refined a comprehensive approach to the coordinate representation of compositional tables based on Aitchison geometry, introducing orthonormal coordinate systems that decompose tables into independent and interaction parts.

Regarding the symmetric structure, \citet{nakamura_symmetry_2024} introduced an Aitchison geometric method for analyzing symmetry in square PTs, demonstrating the potential of this approach.
However, even within this framework, the concept of symmetry often proves too restrictive for real-world data.
Many practical situations exhibit more subtle patterns that deviate from strict symmetry while still maintaining certain structural regularities.
This realization underscores the need for a more refined analysis of table structures that can capture and quantify these subtle deviations from complete symmetry, potentially offering new insights into the underlying processes generating the data.

In this paper, we develop the approach of \citet{nakamura_symmetry_2024} to analyzing symmetric structures in square CTs using Aitchison geometry on the simplex.
We offer a geometric perspective on quasi-symmetry, equivalent to the definition suggested by \citet{caussinus_contribution_1965}, and present the novel idea of geometric marginal homogeneity.
This concept offers an alternative perspective on table structures and is less restrictive than symmetry.
We establish that quasi-symmetric tables and geometric marginal homogeneous tables constitute subspaces within the simplex. Additionally, we show that in Aitchison geometry, the measure of skewness can be decomposed orthogonally into components that represent deviations from quasi-symmetry and geometric marginal homogeneity.

In Section \ref{sec2}, we introduce the Aitchison geometry and its application to the analysis of square PTs.
Section \ref{sec3} defines quasi-symmetry and geometric marginal homogeneity structures, presents their properties, and demonstrates the orthogonal decomposition of square PTs.
Section \ref{sec5} provides an illustrative example using data on unaided distance vision to demonstrate the proposed methodology.
Finally, Section \ref{sec6} concludes the paper and discusses potential future research directions.
All proofs are deferred to the Appendix.

\section{Aitchison Geometry and Decomposition of Probability Tables}\label{sec2}

Before turning to the technical developments, we briefly outline the flow of this section.  
Subsection \ref{subsec2.1} revisits the algebraic operations of Aitchison geometry, the centered log-ratio (clr) transformation, and the resulting Aitchison metric that endows the simplex with a Euclidean metric.
Subsection \ref{subsec2.2} presents the result of \citet{egozcue_independence_2015}, which demonstrates how any $I\times J$ probability table can be orthogonally decomposed into independent and interaction components.
Subsection \ref{subsec2.3} explains the result of \citet{nakamura_symmetry_2024}, which presents a second decomposition into symmetric and skew-symmetric parts.
These two decompositions together provide the geometric toolkit used throughout the remainder of the paper.

\subsection{Fundamentals of Aitchison Geometry}\label{subsec2.1}
Consider an $I\times J$ CT obtained from multinomial sampling.
Let $\mbf{P} = (p_{ij})$ denote the vector of its associated PTs, whose probabilities are assumed to be strictly positive and add up to $1$.
Therefore, the table is regarded as an element of a simplex,
\[
  \mcal{S}^{IJ} = \Set[\bigg]{\mbf{P} = (p_{11}, p_{12}, \ldots, p_{1J}, \ldots, p_{IJ})\in\mbb{R}^{IJ}\given p_{ij}>0, \sum_{i=1}^{I}\sum_{j=1}^{J}p_{ij} = 1}.
\]

To define operations on elements of the simplex, we use the closure operator,
\[
  \mcal{C}\mbf{R} = \left(\frac{r_{11}}{\sum_{i=1}^{I}\sum_{j=1}^{J}r_{ij}}, \ldots, \frac{r_{IJ}}{\sum_{i=1}^{I}\sum_{j=1}^{J}r_{ij}}\right),
\]
where $\mbf{R} \in \Set*{(r_{11}, r_{12}, \ldots, r_{IJ})\in\mbb{R}^{IJ}\given r_{ij} > 0}$.
\citet{aitchison_statistical_1982, 10.5555/17272} introduced perturbation as an addition operation and powering as scalar multiplication in compositional data analysis.
For two PTs $\mbf{P}$ and $\mbf{Q}$, and a real number $\alpha$, the perturbation and powering are respectively given as
\begin{align*}
  &\mbf{P}\oplus\mbf{Q} = \mcal{C}(p_{11}q_{11},\ldots p_{IJ}q_{IJ}), \\
  &\alpha\odot\mbf{P} = \mcal{C}(p_{11}^{\alpha},\ldots,p_{IJ}^{\alpha}).
\end{align*}
The simplex $\mcal{S}^{IJ}$ equipped with $\oplus$ and $\odot$ is an $(IJ-1)$-dimensional vector space \citep{billheimer_statistical_2001, pawlowsky-glahn_geometric_2001}.

For any PT $\mbf{P}\in\mcal{S}^{IJ}$, the $I\times J$ table, denoted by $\clr(\mbf{P})$, whose entries
\[
  \clr_{ij}(\mbf{P}) = \log\frac{p_{ij}}{g(\mbf{P})},\quad g(\mbf{P})=\left(\prod_{i=1}^{I}\prod_{j=1}^{J}p_{ij}\right)^{1/IJ},
\]
is called the centered log-ratio (clr) transformation of $\mbf{P}$.
This transformation is an isomorphism from $\mcal{S}^{IJ}$ to $U\subset\mbb{R}^{IJ}$ where tables have elements adding up to $0$.
Therefore, $\clr(\mbf{P})$ is not full rank and spans a subspace of $(IJ-1)$-dimensional hyperplane called clr-plane in $\mbb{R}^{IJ}$.
The inverse of this transformation is given as
\[
  \mbf{P} = \mcal{C}\exp[\clr(\mbf{P})],
\]
where $\exp$ operates componentwise.

Using clr transformation, the Aitchison inner product $\langle\cdot, \cdot\rangle_{A}$, the norm $\|\cdot\|_{A}$, and the distance $d_{A}(\cdot, \cdot)$, are defined as
\[
  \langle\cdot, \cdot\rangle_{A} = \langle\clr(\cdot), \clr(\cdot)\rangle,\ \|\cdot\|_{A} = \|\clr(\cdot)\|,\ d_{A}(\cdot, \cdot) = d(\clr(\cdot), \clr(\cdot)),
\]
where $\langle\cdot, \cdot\rangle$, $\|\cdot\|$, and $d(\cdot, \cdot)$ denote the ordinary Euclidean inner product, norm, and distance in $\mbb{R}^{IJ}$, respectively.
In addition to perturbation and powering, the simplex equipped with this metric is a $(IJ-1)$-dimensional Euclidean space \citep{billheimer_statistical_2001, pawlowsky-glahn_geometric_2001}.
Hence, it is possible to sum one table with another, compute the distance and norm of PTs, and orthogonally project a PT onto a given linear subspace.

It is worth noting that this geometric structure can also be interpreted within the context of information geometry.
\citet{erb_power_2023} observed an important connection: the clr-plane, which is central to Aitchison geometry, corresponds to the tangent space of the simplex in information-geometric terms.
From an information-geometric perspective, the straight lines in the clr-plane correspond to e-geodesics in the original simplex.
For two PTs $\mbf{P}$ and $\mbf{Q}$ in the simplex, and a parameter $\lambda \in [0,1]$, the e-geodesic connecting them can be expressed using perturbation and powering as
\[
  \gamma_e(\lambda) = (1-\lambda) \odot \mbf{P} \oplus \lambda \odot \mbf{Q} = \mathcal{C}(p_{11}^{1-\lambda}q_{11}^{\lambda}, \ldots, p_{IJ}^{1-\lambda}q_{IJ}^{\lambda}).
\]
This expression shows how movement along e-geodesics can be represented through the operations of Aitchison geometry.
In the clr-plane, this corresponds to a linear combination,
\[
  \clr(\gamma_e(\lambda)) = (1-\lambda)\clr(\mbf{P}) + \lambda\clr(\mbf{Q}).
\]
While our subsequent analysis primarily employs Aitchison geometry, this connection highlights how the same compositional relationships could alternatively be viewed through the dual structure of information geometry.

In Aitchison geometry, due to the closure operation, only relative values carry meaningful information.
Consequently, our analysis focuses on the ratios between cell probabilities, disregarding normalization constants.
This approach captures the essential structural information of the table while remaining invariant to scaling.
In subsequent analyses, we consider primarily these probability ratios, as they contain all relevant information in this geometric context.

\subsection{Decomposition into Independent and Interaction Components}\label{subsec2.2}
\citet{egozcue_compositional_2008} presented projections that explain the ratios between different rows and columns.
This leads to projections onto subspaces $S^{IJ}(row)$ and $S^{IJ}(col)$, which represent row and column subspaces, respectively.
These projections result in
\[
  row(\mbf{P}) = \mcal{C}
  \begin{pmatrix}
    g(row_{1}[\mbf{P}]) & g(row_{1}[\mbf{P}]) & \cdots & g(row_{1}[\mbf{P}]) \\
    g(row_{2}[\mbf{P}]) & g(row_{2}[\mbf{P}]) & \cdots & g(row_{2}[\mbf{P}]) \\
    \vdots & \vdots & \ddots & \vdots \\
    g(row_{I}[\mbf{P}]) & g(row_{I}[\mbf{P}]) & \cdots & g(row_{I}[\mbf{P}])
  \end{pmatrix},
\]
and
\[
  col(\mbf{P}) = \mcal{C}
  \begin{pmatrix}
  g(col_{1}[\mbf{P}]) & g(col_{2}[\mbf{P}]) & \cdots & g(col_{J}[\mbf{P}]) \\
  g(col_{1}[\mbf{P}]) & g(col_{2}[\mbf{P}]) & \cdots & g(col_{J}[\mbf{P}]) \\
  \vdots & \vdots & \ddots & \vdots \\
  g(col_{1}[\mbf{P}]) & g(col_{2}[\mbf{P}]) & \cdots & g(col_{J}[\mbf{P}])
  \end{pmatrix},
\]
where $g(row_{i}[\mbf{P}])$ and $g(col_{j}[\mbf{P}])$ denote the geometric mean of elements in the $i$-th row and $j$-th column of $\mbf{P}$, respectively.
According to \citet{egozcue_compositional_2008}, the projections $row(\mbf{P})$ and $col(\mbf{P})$ are orthogonal to each other.
Here, $row(\mbf{P})$ and $col(\mbf{P})$ can be interpreted as geometric marginals.

From an information-geometric perspective, these geometric marginals represent the projection of the PT onto e-flat subspaces in the simplex.
When viewed in the clr-plane, these projections become linear operations
\[
  \clr_{ij}(row(\mbf{P})) = \log\frac{g(row_i[\mbf{P}])}{g(\mbf{P})} = \frac{1}{J}\sum_{j=1}^{J}\clr_{ij}(\mbf{P}),
\]
and
\[
  \clr_{ij}(col(\mbf{P})) = \log\frac{g(col_j[\mbf{P}])}{g(\mbf{P})} = \frac{1}{I}\sum_{i=1}^{I}\clr_{ij}(\mbf{P}).
\]
While traditional arithmetic marginals correspond to m-flat structures in information geometry, geometric marginals form their natural counterparts as e-flat structures.
In the framework of Aitchison geometry, which emphasizes relative relationships between cells, geometric marginals provide a more natural characterization of the underlying structure.

Geometric marginals are key in analyzing independence between categorical variables as they coincide with arithmetic marginals under independence \citep{egozcue_independence_2015}.
Therefore, if row and column variables are independent, the geometric marginals $row(\mbf{P})$ and $col(\mbf{P})$ alone suffice to reconstruct the original PT.
The resulting $I \times J$ PT, obtained as a perturbation of these two projections, is called the independent PT,
\[
  \mbf{P}_{\ind} = row(\mbf{P}) \oplus col(\mbf{P}),
\]
with related cells
\[
  p_{ij}^{\ind}\propto\left(\prod_{k=1}^I \prod_{l=1}^J p_{kj}p_{il}\right)^{1/IJ}.
\]
\citet{egozcue_compositional_2008, egozcue_independence_2015} demonstrated that the set of independent PTs constitutes a linear subspace, denoted as $\mcal{S}^{IJ}_{\ind}$.
Because the dimensions of $S^{IJ}(row)$ and $S^{IJ}(col)$ are $I-1$ and $J-1$, respectively, the dimension of $S^{IJ}_{\ind}$ equals $I+J-2$.

Leveraging the Hilbert projection theorem, any PT $\mbf{P}$ can be uniquely decomposed into two orthogonal components: the closest independent PT $\mbf{P}_{\ind} \in \mcal{S}^{IJ}_{\ind}$ and the interaction PT $\mbf{P}_{\inte} \in \mcal{S}^{IJ}_{\inte}$.
Here, $\mcal{S}^{IJ}_{\inte}$ denotes the orthogonal complement of $\mcal{S}^{IJ}_{\ind}$.
This decomposition is expressed as
\[
  \mbf{P} = \mbf{P}_{\ind}\oplus\mbf{P}_{\inte},\ \langle\mbf{P}_{\ind}, \mbf{P}_{\inte}\rangle_{A} = 0.
\]
Cells of $\mbf{P}_{\inte}$ can be computed from the original PT $\mbf{P}$ by
\[
  p_{ij}^{\inte}\propto\left(\prod_{k=1}^{I} \prod_{l=1}^{J} \frac{p_{ij}}{p_{kj}p_{il}}\right)^{1/IJ}.
\]
The dimension of $S^{IJ}_{\inte}$ equals $IJ-1-(I+J-2) = (I-1)(J-1)$.
Based on this framework, \citet{egozcue_independence_2015} introduced a measure of dependence termed \textit{simplicial deviance}, defined as $\|\mbf{P}_{\inte}\|_{A}^{2}$, and proposed a method for testing independence using this measure.

\subsection{Decomposition into Symmetric and Skew-Symmetric Components}\label{subsec2.3}
In the case of \( I \times I \) square contingency tables, where rows and columns have the same categories, \citet{nakamura_symmetry_2024} showed, similarly to the case of independence, that the set of symmetric PTs forms a linear subspace $\mcal{S}^{I^{2}}_{\sym}$ of the simplex $\mcal{S}^{I^{2}}$.
The dimension of $\mcal{S}^{I^{2}}_\sym$ equals $(I-1)(I+2)/2$.
They demonstrated the orthogonal decomposition of a square PT $\mbf{P}$ into the closest symmetric PT $\mbf{P}_{\sym}\in\mcal{S}^{I^{2}}_{\sym}$ and the skew-symmetric PT $\mbf{P}_{\skw}\in\mcal{S}^{I^{2}}_{\skw}$, which is an orthogonal complement of $\mcal{S}^{I^{2}}_{\sym}$.
The dimension of $\mcal{S}^{I^{2}}_{\skw}$ equals $I(I-1)/2$.
This decomposition can be expressed as
\[
  \mbf{P} = \mbf{P}_{\sym}\oplus\mbf{P}_{\skw},\ \langle\mbf{P}_{\sym}, \mbf{P}_{\skw}\rangle_{A} = 0.
\]

They showed that the nearest symmetric PT is obtained by the geometric mean of symmetric cells,
\[
  \mbf{P}_{\sym} = 0.5\odot\big(\mbf{P}\oplus T(\mbf{P})\big),
\]
with related cells
\[
  p_{ij}^{\sym}\propto\left(p_{ij}p_{ji}\right)^{1/2},
\]
where $T(\cdot)$ is a vector transposition corresponding to a matrix transpose.
It is noteworthy that this result coincides with the BAN estimator proposed by \citet{ireland_symmetry_1969}.
Their estimator is the minimum point of reverse Kullback-Leibler divergence from population probabilities to $\mcal{S}^{I^{2}}_{\sym}$.
This procedure corresponds to an e-projection, as it projects the population probability onto the $\mcal{S}^{I^{2}}_{\sym}$ along the e-geodesic.
Consequently, this result coincides with the orthogonal projection that minimizes the Aitchison distance, as the clr-transformed space serves as the tangent space of the e-geodesic.

Cells of $\mbf{P}_{\skw}$ can be computed from the original $\mbf{P}$ by
\[
  p_{ij}^{\skw}\propto\left(\frac{p_{ij}}{p_{ji}}\right)^{1/2}.
\]
To quantify the degree of skewness in a PT, they proposed a measure called \textit{simplicial skewness} defined as

\[
  E^{2}(\mbf{P}) = \|\mbf{P}_{\skw}\|_{A}^{2}.
\]
This measure provides a scalar value representing the overall skewness in the table, analogous to the simplicial deviance from independence.
To assess each cell's contribution to skewness, they proposed an array, termed the \textit{skewness array}, with $(i,j)$-th element defined as $\mathrm{sgn}\big(\clr_{ij}(\mbf{P}_{\skw})\big)\big(\clr_{ij}(\mbf{P}_{\skw})\big)^2/E^2(\mbf{P})$.
In this skewness array, each entry represents the signed contribution of the corresponding cell to the simplicial skewness.

This geometric perspective on symmetry complements the independence-interaction decomposition discussed earlier, providing a comprehensive framework for analyzing the structure of square CTs.
In the following sections, we will explore how these concepts can be combined and extended to study more complex structures such as quasi-symmetry and geometric marginal homogeneity.

\section{Quasi-symmetry and geometric marginal homogeneity}\label{sec3}
\citet{egozcue_interpretable_2021} first introduced the geometric approach to studying symmetry and skew-symmetry in square tables, in the context of tables with positive components using the log-geometric approach, which is a non-closed Aitchison geometry.
In this study, we consider PTs under the Aitchison geometry on the simplex following a similar approach.
Specifically, we consider four subspaces of the simplex $\mcal{S}^{I^{2}}$: $\mcal{S}^{I^{2}}_{\ind}$, $\mcal{S}^{I^{2}}_{\inte}$, $\mcal{S}^{I^{2}}_{\sym}$, and $\mcal{S}^{I^{2}}_{\skw}$, and define four new subspaces as their intersections.

\begin{definition}[four subspaces of $\mcal{S}^{I^{2}}$]
  \label{def1}
  Let $\mcal{S}^{I^{2}}_{\ind}$, $\mcal{S}^{I^{2}}_{\inte}$, $\mcal{S}^{I^{2}}_{\sym}$, and $\mcal{S}^{I^{2}}_{\skw}$ be subspaces of $\mcal{S}^{I^{2}}$, where $\mcal{S}^{I^{2}}_{\ind}$ is the set of all independence tables, $\mcal{S}^{I^{2}}_{\inte}$ is the set of all interaction tables, $\mcal{S}^{I^{2}}_{\sym}$ is the set of all symmetric tables, and $\mcal{S}^{I^{2}}_{\skw}$ is the set of all skew-symmetric tables in $\mcal{S}^{I^{2}}$. We define the following subspaces as their intersections,
  \begin{align*}
    &\mcal{S}^{I^{2}}_{\syind} = \mcal{S}^{I^{2}}_{\ind}\cap\mcal{S}^{I^{2}}_{\sym},\
    \mcal{S}^{I^{2}}_{\skind} = \mcal{S}^{I^{2}}_{\ind}\cap\mcal{S}^{I^{2}}_{\skw},\\
    &\mcal{S}^{I^{2}}_{\syint} = \mcal{S}^{I^{2}}_{\inte}\cap\mcal{S}^{I^{2}}_{\sym},\
    \ \mcal{S}^{I^{2}}_{\skint} = \mcal{S}^{I^{2}}_{\inte}\cap\mcal{S}^{I^{2}}_{\skw}.
  \end{align*}
\end{definition}

The four subspaces form the basis for the orthogonal decomposition of any square PT in $\mcal{S}^{I^{2}}$.
This decomposition allows us to express any PT as a unique combination of its symmetric independence, skew-symmetric independence, symmetric interaction, and skew-symmetric interaction components.
The following theorem formalizes this orthogonal decomposition.

\begin{theorem}[Orthogonal decomposition]
  \label{thm1}
  Let $\mbf{P}$ be an $I\times I\in\mcal{S}^{I^{2}}$ table. $\mbf{P}$ can be expressed in a unique way as
  \begin{align*}
    \mbf{P} = \mbf{P}_{\syind}\oplus\mbf{P}_{\skind}\oplus\mbf{P}_{\syint}\oplus\mbf{P}_{\skint},
  \end{align*}
  where $\mbf{P}_{\syind}\in\mcal{S}^{I^{2}}_{\syind}$, $\mbf{P}_{\skind}\in\mcal{S}^{I^{2}}_{\skind}$, $\mbf{P}_{\syint}\in\mcal{S}^{I^{2}}_{\syint}$, and $\mbf{P}_{\skint}\in\mcal{S}^{I^{2}}_{\skint}$.
\end{theorem}

This orthogonal decomposition provides a framework for analyzing the structure of square PTs in terms of independence, interaction, symmetry, and skew-symmetry.
In the following sections, we will explore how this decomposition can be utilized to study quasi-symmetry and geometric marginal homogeneity, two important properties of square PTs.

\subsection{Quasi-symmetric PT}\label{sec3.1}
Quasi-symmetry (QS) is a less restrictive statistical property of square PTs compared to complete symmetry. 
While symmetry requires that cell probabilities be symmetric, QS allows for differences in marginals and only requires the symmetry of local odds ratios.
QS enables researchers to better understand and interpret complex relationships within categorical data, even in cases where the data do not fully adhere to symmetry requirements.
Based on prior research \citep{egozcue_compositional_2008, egozcue_independence_2015}, we can intuitively define quasi-symmetric PT as follows.
\begin{definition}[Quasi-symmetric PT]
  \label{def:def2}
  Let $\mbf{P}\in\mcal{S}^{I^{2}}$ be an $I\times I$ table with decomposition $\mbf{P}=\mbf{P}_{ind}\oplus\mbf{P}_{\inte}$.
  $\mbf{P}$ is a quasi-symmetric PT if
  \begin{align*}
    \mbf{P}_{\inte} = T(\mbf{P}_{\inte}).
  \end{align*}
\end{definition}
We can prove that the above definition is equivalent to that of QS proposed by \citet{caussinus_contribution_1965}.
\begin{theorem}
  \label{thm2}
  Quasi-symmetry given in Definition \ref{def:def2} is equivalent to the quasi-symmetric structure defined by \eqref{eq1}.
\end{theorem}
This equivalence emphasizes that both the traditional definition of QS and our Aitchison geometry-based definition represent e-flat structures in their respective geometric frameworks, highlighting the natural correspondence between these characterizations.
The following theorem shows that the set of all quasi-symmetric PTs forms a linear subspace of the simplex $\mcal{S}^{I^{2}}$ and specifies its dimensionality.

\begin{theorem}[Subspace of quasi-symmetric PT]
  \label{thm3}
  Let $\mcal{S}_{\qs}^{I^{2}}$ be the set of quasi-symmetric PTs.
  $\mcal{S}_{\qs}^{I^{2}}$ is a linear subspace of $\mcal{S}^{I^{2}}$.
  Its dimension is $(I-1)(I+4)/2$.
\end{theorem}

The fact that the set of quasi-symmetric PTs forms a linear subspace allows us to define the orthogonal projection of any PT onto this subspace.
The following theorem provides an explicit formula for this projection, which can be interpreted as the closest quasi-symmetric PT to a given PT in terms of the Aitchison distance.

\begin{theorem}[orthogonal projection]
  \label{thm4}
  Let $\mbf{P}\in\mcal{S}^{I^{2}}$ be an $I\times I$ table.
  The closest quasi-symmetric PT to $\mbf{P}$ is $\mbf{P}_{\qs} = \mbf{P}_{\ind}\oplus\Big(0.5 \odot \big(\mbf{P}_{\inte} \oplus T(\mbf{P}_{\inte})\big)\Big)$ in the sense of the Aitchison distance in $\mcal{S}^{I^{2}}$.
\end{theorem}

The orthogonal projection onto the $\mcal{S}_{\qs}^{I^{2}}$ decomposes any PT into two components: the quasi-symmetric component $\mbf{P}_{\qs}$ and the skew-symmetric interaction component $\mbf{P}_{\skint}=\mbf{P}\ominus\mbf{P}_{\qs}$, which represents the deviation from QS.
Cells of $\mbf{P}_{\qs}$ can be computed from the original $\mbf{P}$ by
\[
  p_{ij}^{\qs}\propto\left(p_{ij}p_{ji}\frac{g(row_{i}[\mbf{P}])g(col_{j}[\mbf{P}])}{g(row_{j}[\mbf{P}])g(col_{i}[\mbf{P}])}\right)^{1/2},
\]
and cells of $\mbf{P}_{\skint}$ can be computed from the original $\mbf{P}$ by
\[
 p_{ij}^{\skint}\propto\left(\frac{p_{ij}}{p_{ji}}\frac{g(row_{j}[\mbf{P}])g(col_{i}[\mbf{P}])}{g(row_{i}[\mbf{P}])g(col_{j}[\mbf{P}])}\right)^{1/2}.
\]

It is important to note that the geometric marginals of this quasi-symmetric PT remain unchanged from those of the original PT.
This preservation of geometric marginals is a key characteristic of the QS.
By maintaining the original marginal structure, QS isolates and focuses on the symmetry of the interaction effects alone.
This feature highlights QS's ability to examine the symmetric nature of associations between variables independently of their marginals, making it particularly useful in analyzing CTs where we want to study interaction effects separate from the marginal structures.
The magnitude of this deviation can be quantified using the simplicial quasi-skewness measure, defined as follows.

\begin{definition}[Simplicial quasi-skewness]
    \label{def3}
    Let $\mbf{P}\in\mcal{S}^{I^{2}}$ be an $I\times I$ table with the orthogonal decomposition $\mbf{P}=\mbf{P}_{\qs}\oplus\mbf{P}_{\skint}$.
    The simplicial quasi-skewness, which measures the degree of departure from quasi-symmetry, is defined as $Q^{2}(\mbf{P})=\|\mbf{P}_{\skint}\|_{A}^{2}$, with $0\leq Q^{2}(\mbf{P})<+\infty$.
\end{definition}

\subsection{Geometric Marginal Homogeneous PT}\label{sec3.2}
Geometric marginal homogeneity (GMH) is a novel concept in the analysis of square CTs, introduced within the framework of Aitchison geometry on the simplex.
While traditional marginal homogeneity focuses on the equality of arithmetic marginals, GMH considers the equality of geometric marginals.
This property implies a balance between the corresponding groups of rows and columns in terms of their geometric means \citep{egozcue_groups_2005}.
GMH is less restrictive than complete symmetry, as it allows for skew-symmetries in the interaction structure while maintaining homogeneity in the geometric marginals.
By focusing on the geometric structure of the marginals, GMH provides insights into the relative relationships between categories, which are particularly relevant in the context of compositional data analysis.

\begin{definition}[Geometric marginal homogeneity]
  \label{def4}
  Let $\mbf{P}\in\mcal{S}^{I^{2}}$ be an $I\times I$ table with orthogonal decomposition $\mbf{P}=\mbf{P}_{\ind}\oplus\mbf{P}_{\inte}$.
  $\mbf{P}$ is a geometric marginal homogeneity PT if
  \begin{align*}
    \mbf{P}_{\ind} = T(\mbf{P}_{\ind}).
  \end{align*}
\end{definition}

This definition implies that for a PT with GMH, the independent component $\mbf{P}_{\ind}$ is symmetric.
In other words, row and column geometric marginals of $\mbf{P}$ are equal.
While traditional marginal homogeneity corresponds to an m-flat structure in information geometry, GMH represents its natural counterpart as an e-flat structure, providing a coherent framework that is consistent with the compositional approach to PTs.

We can characterize the GMH subspace as follows.
\begin{theorem}[Subspace of geometric marginal homogeneity PT]
\label{thm5}
  Let $\mcal{S}_{\gmh}^{I^{2}}$ be the set of geometric marginal homogeneity PTs.
  $\mcal{S}_{\gmh}^{I^{2}}$ is a linear subspace of $\mcal{S}^{I^{2}}$. Its dimension is $I(I-1)$.
\end{theorem}
Similar to the case of quasi-symmetric PT, this is a crucial result as it allows us to apply orthogonal projection within the framework of Aitchison geometry.
The following theorem provides an explicit formula for this projection.
\begin{theorem}[orthogonal projection]
  \label{thm6}
  Let $\mbf{P}\in\mcal{S}^{I^{2}}$ be $I\times I$ table. The closest geometric marginal homogeneous PT to $\mbf{P}$ is $\mbf{P}_{\gmh} = \Big(0.5 \odot \big(\mbf{P}_{\ind} \oplus T(\mbf{P}_{\ind})\big)\Big)\oplus\mbf{P}_{\inte}$ in the sense of the Aitchison distance in $\mcal{S}^{I^{2}}$.
\end{theorem}
This theorem provides us with a powerful tool for analyzing the GMH of any given PT.
Cells of $\mbf{P}_{\gmh}$ can be computed from the original $\mbf{P}$ by
\[
  p_{ij}^{\gmh}\propto p_{ij}\left(\frac{g(row_{j}[\mbf{P}])g(col_{i}[\mbf{P}])}{g(row_{i}[\mbf{P}])g(col_{j}[\mbf{P}])}\right)^{1/2}.
\]
and cells of $\mbf{P}_{\skind}$ can be computed from the original $\mbf{P}$ by
\[
 p_{ij}^{\skind}\propto\left(\frac{g(row_{i}[\mbf{P}])g(col_{j}[\mbf{P}])}{g(row_{j}[\mbf{P}])g(col_{i}[\mbf{P}])}\right)^{1/2}.
\]
In particular, the geometric marginals of $\mbf{P}_{\gmh}$ are the geometric means of the row and column geometric marginals of the original $\mbf{P}$.
The interaction component of $\mbf{P}_{\gmh}$ is identical to that of $\mbf{P}$.
To fully understand the degree of departure from GMH in a PT, we need a measure that captures the deviation from this property.
The concept of simplicial geometric marginal heterogeneity effectively addresses this need, which the following definition formalizes.
\begin{definition}[Simplicial geometric marginal heterogeneity]
    \label{def5}
    Let $\mbf{P}\in\mcal{S}^{I^{2}}$ be an $I\times I$ table with the orthogonal decomposition, $\mbf{P}=\mbf{P}_{\gmh}\oplus\mbf{P}_{\skind}$.
    The simplicial geometric marginal heterogeneity, which measures the degree of geometric marginal heterogeneity, is defined as $M^{2}(\mbf{P})=\|\mbf{P}_{\skind}\|_{A}^{2}$, with $0\leq M^{2}(\mbf{P})<+\infty$.
\end{definition}
Interestingly, this measure of geometric marginal heterogeneity is not isolated but relates to other important concepts in the analysis of PTs. Specifically, it forms part of a larger decomposition of the overall skewness of a PT.
This relationship is formalized in the following theorem, which demonstrates how the simplicial skewness of a PT can be orthogonally decomposed into components representing quasi-skewness and geometric marginal heterogeneity.
\begin{theorem}[Orthogonal decomposition of simplicial skewness]
  \label{thm7}
  Given a PT $\mbf{P}$, its simplicial skewness $E^{2}(\mbf{P})$ can be decomposed into simplicial quasi-skewness $Q^{2}(\mbf{P})$ and geometric marginal heterogeneity $M^{2}(\mbf{P})$, that is,
  \begin{align}
    \label{eq:eq2}
    E^{2}(\mbf{P}) = Q^{2}(\mbf{P}) + M^{2}(\mbf{P}).
  \end{align}
\end{theorem}
The orthogonal decomposition of simplicial skewness into quasi-skewness and geometric marginal heterogeneity provides a powerful tool for analyzing the structure of square PTs.
This decomposition allows separate quantification of two distinct aspects of a PT's structure, offering new insights into skewness in square CTs.
The orthogonality of these components enables the attribution of the observed skewness to specific structural characteristics, which is particularly valuable in statistical analysis.

Detailed description of skewness can be obtained using the following definition of cell quasi-skewness and geometric marginal heterogeneity.
\begin{definition}
  \label{def:6}
  Let $\mbf{P}_{\skint}\in\mcal{S}^{I^{2}}_{\skint}$ and $\mbf{P}_{\skind}\in\mcal{S}^{I^{2}}_{\skind}$ be an $I\times I$ skew-symmetric interaction table.
  The coefficient $\clr_{ij}(\mbf{P}_{\skint})$ and $\clr_{ij}(\mbf{P}_{\skind})$ is called the $(i, j)$-th cell quasi-skewness and cell geometric marginal heterogeneity. 
\end{definition}
Simplicial quasi-skewness is the square sum of cell quasi-skewnesses and simplicial geometric marginal heterogeneity is the square sum of cell geometric marginal heterogeneities:
\begin{align*}
  Q^{2}(\mbf{P}) &= \sum_{i=1}^{I}\sum_{j=1}^{I}\big(\clr_{ij}(\mbf{P}_{\skint})\big)^{2}, \\
  M^{2}(\mbf{P}) &= \sum_{i=1}^{I}\sum_{j=1}^{I}\big(\clr_{ij}(\mbf{P}_{\skind})\big)^{2}.
\end{align*}

The cell-wise contributions to quasi-skewness and geometric marginal heterogeneity are presented as $I \times I$ tables, which we call quasi-skewness array and geometric marginal heterogeneity array, respectively.
In these arrays, the entries are the signed proportions or percents of the simplicial quasi-skewness and simplicial geometric marginal heterogeneity, calculated as $\mathrm{sgn}\big(\clr_{ij}(\mbf{P}_{\skint})\big)\big(\clr_{ij}(\mbf{P}_{\skint})\big)^2/Q^2(\mbf{P})$ for the quasi-skewness array and $\mathrm{sgn}\big(\clr_{ij}(\mbf{P}_{\skind})\big)\big(\clr_{ij}(\mbf{P}_{\skind})\big)^2/M^2(\mbf{P})$ for the geometric marginal heterogeneity array. 
The strength of the quasi-skewness and geometric marginal heterogeneity is quantified, and the direction is illustrated by the sign, indicating whether its probability is a surplus or deficit compared with the corresponding quasi-symmetric PT or geometric marginal homogeneous PT, respectively.
These arrays provide a detailed view of how each cell contributes to the quasi-skewness and geometric marginal heterogeneity, allowing for a more comprehensive understanding of the structure of skewness in the PTs.
By examining these arrays alongside the overall skewness array, we can gain insights into the specific sources and patterns of skewness that might not be apparent from the aggregate measures alone.

\section{Illustrative example}\label{sec5}
\begin{table}[t]
\caption{Data of unaided distance vision of 7477 women aged 30-39 \citep{stuart_estimation_1953}}\label{tab:stuart_data}
\begin{tabular}{|c|cccc|c|}
    \toprule
    \multirow{2}{*}{\diagbox{Right eye}{Left eye}} & Highest & Second & Third & Lowest & Total \\
     & grade & grade & grade & grade & \\
    \hline
    Highest grade & 1520 & 266 & 124 & 66 & 1976 \\
    Second grade & 234 & 1512 & 432 & 78 & 2256 \\
    Third grade & 117 & 362 & 1772 & 205 & 2456 \\
    Lowest grade & 36 & 82 & 179 & 492 & 789 \\
    \hline
    Total & 1907 & 2222 & 2507 & 841 & 7477 \\
    \hline
  \end{tabular}
  \caption{Sample proportions of unaided distance vision data}
  \label{tab:stuart_proportions}
  \begin{tabular}{|c|cccc|c|}
    \hline
    \multirow{2}{*}{\diagbox{Right eye}{Left eye}} & Highest & Second & Third & Lowest & Geometric \\
     & grade & grade & grade & grade & margin\\
    \hline
    Highest grade & 0.2033 & 0.0356 & 0.0166 & 0.0088 & 0.2285 \\
    Second grade & 0.0313 & 0.2022 & 0.0578 & 0.0104 &0.3149 \\
    Third grade & 0.0156 & 0.0484 & 0.2370 & 0.0274 & 0.3356 \\
    Lowest grade & 0.0048 & 0.0110 & 0.0239 & 0.0658 & 0.1210 \\
    \hline
    Geometric margin & 0.1893 & 0.3181 & 0.3474 & 0.1452 & 1.0000 \\
    \hline
  \end{tabular}
\end{table}
To illustrate the application and effectiveness of our proposed methodology, we will use a dataset from \citet{stuart_estimation_1953} that describes the unaided distance vision of 7477 women aged 30-39.
The data classifies vision into four categories for each eye: Highest, Second, Third, and Lowest grades.
Table \ref{tab:stuart_data} shows the raw counts of unaided distance vision for 7477 women, while Table \ref{tab:stuart_proportions} presents the corresponding sample proportions $\hat{\mbf{P}}$.
This structure of the $4\times4$ square table is ideal for our analysis, allowing us to examine both the symmetry of the vision grades between the left and right eyes and the potential skew-symmetry that may exist.
The analysis of Stuart's eyesight data using our proposed geometric approach may reveal subtle skewnesses in the distribution of vision grades between left and right eyes, potentially contributing to our understanding of binocular vision patterns in this population.

The Aitchison norm of the original table $\|\hat{\mbf{P}}\|_{A}$ is $20.560$.
As shown in \citet{nakamura_symmetry_2024}, Tables \ref{tab:stuart_symmetry} and \ref{tab:stuart_skew-symmetry} show the closest symmetric PT and the corresponding skew-symmetric PT for the unaided distance vision data.
The norm of its symmetric component $\|\hat{\mbf{P}}_{\sym}\|_{A}$ is $20.341$. 
\begin{table}[t]
  \centering
  \caption{Nearest symmetric PT of Table \ref{tab:stuart_proportions} in the sense of Aitchison distance \citep{nakamura_symmetry_2024}}
  \label{tab:stuart_symmetry}
  \begin{tabular}{|c|cccc|c|}
    \hline
    \multirow{2}{*}{\diagbox{Right eye}{Left eye}} & Highest & Second & Third & Lowest & Geometric \\
     & grade & grade & grade & grade & margin\\
    \hline
    Highest grade & 0.2036 & 0.0334 & 0.0161 & 0.0065 & 0.2596 \\
    Second grade & 0.0334 & 0.2025 & 0.0530 & 0.0107 & 0.2996 \\
    Third grade & 0.0161 & 0.0530 & 0.2373 & 0.0257 & 0.3320 \\
    Lowest grade & 0.0065 & 0.0107 & 0.0257 & 0.0659 & 0.1088 \\
    \hline
    Geometric margin & 0.2596 & 0.2996 & 0.3320 & 0.1088 & 1.0000 \\
    \hline
  \end{tabular}
  \caption{Skew-symmetric PT orthogonal to Table \ref{tab:stuart_symmetry}  \citep{nakamura_symmetry_2024}}
  \label{tab:stuart_skew-symmetry}
  \begin{tabular}{|c|cccc|c|}
    \hline
    \multirow{2}{*}{\diagbox{Right eye}{Left eye}} & Highest & Second & Third & Lowest & Geometric \\
     & grade & grade & grade & grade & margin\\
    \hline
    Highest grade & 0.0621 & 0.0662 & 0.0639 & 0.0840 & 0.2762 \\
    Second grade & 0.0582 & 0.0621 & 0.0678 & 0.0605 & 0.2486 \\
    Third grade & 0.0603 & 0.0568 & 0.0621 & 0.0664 & 0.2456 \\
    Lowest grade & 0.0458 & 0.0636 & 0.0580 & 0.0621 & 0.2296 \\
    \hline
    Geometric margin & 0.2264 & 0.2487 & 0.2518 & 0.2731 & 1.0000 \\
    \hline
  \end{tabular}
\end{table}
The simplicial skewness $E^2(\hat{\mbf{P}})$, which measures the overall skewness of the table, is $0.219$.

To further investigate the nature of this skewness, we decompose it into quasi-skewness and geometric marginal heterogeneity.
\begin{table}[t]
  \caption{Nearest quasi-symmetric PT of Table \ref{tab:stuart_proportions} in the sense of Aitchison distance}
  \label{tab:stuart_quasi-symmetry}
  \centering
  \begin{tabular}{|c|cccc|c|}
    \hline
    \multirow{2}{*}{\diagbox{Right eye}{Left eye}} & Highest & Second & Third & Lowest & Geometric \\
     & grade & grade & grade & grade & margin\\
    \hline
    Highest grade & $0.2034$ & $0.0369$ & $0.0180$ & $0.0079$ & $0.2285$ \\
    Second grade & $0.0302$ & $0.2023$ & $0.0536$ & $0.0117$ & $0.3149$ \\
    Third grade & $0.0144$ & $0.0523$ & $0.2371$ & $0.0276$ & $0.3356$ \\
    Lowest grade & $0.0054$ & $0.0098$ & $ 0.0238$ & $0.0658$ & $0.1210$ \\
    \hline
    Geometric margin & $0.1893$ & $0.3181$ & $0.3474$ & $0.1452$ & $1.0000$ \\
    \hline
  \end{tabular}
  \caption{Skew-symmetric interaction PT orthogonal to Table \ref{tab:stuart_quasi-symmetry}}
  \label{tab:stuart_quasi-skew-symmetry}
  \begin{tabular}{|c|cccc|c|}
    \hline
    \multirow{2}{*}{\diagbox{Right eye}{Left eye}} & Highest & Second & Third & Lowest & Geometric \\
     & grade & grade & grade & grade & margin\\
    \hline
    Highest grade & $0.0623$ & $0.0602$ & $0.0574$ & $0.0701$ & $0.2500$ \\
    Second grade & $0.0582$ & $0.0623$ & $0.0678$ & $0.0605$ & $0.2500$ \\
    Third grade & $ 0.0677$ & $0.0578$ & $0.0623$ & $0.0620$ & $0.2500$ \\
    Lowest grade & $0.0554$ & $0.0637$ & $0.0627$ & $0.0623$ & $0.2500$ \\
    \hline
    Geometric margin & $0.2500$ & $0.2500$ & $0.2500$ & $0.2500$ & $1.0000$ \\
    \hline
  \end{tabular}
\end{table}
Table \ref{tab:stuart_quasi-symmetry} shows the nearest quasi-symmetric PT and its local odds ratio matrix is
\[
\begin{pmatrix}
     36.9231 & 0.5417 & 0.4997 \\
     0.5417 & 17.1326 & 0.5345 \\
     0.4997 & 0.5345 & 23.7687
\end{pmatrix}.
\]
Table \ref{tab:stuart_quasi-symmetry} shows the skew-symmetric interaction PT orthogonal to Table \ref{tab:stuart_quasi-skew-symmetry}.
\begin{table}[t]
  \caption{Nearest geometric marginal homogeneous PT of Table \ref{tab:stuart_proportions} in the sense of Aitchison distance}
  \label{tab:stuart_geometric marginal homogeneity}
  \centering
  \begin{tabular}{|c|cccc|c|}
    \hline
    \multirow{2}{*}{\diagbox{Right eye}{Left eye}} & Highest & Second & Third & Lowest & Geometric \\
     & grade & grade & grade & grade & margin\\
    \hline
    Highest grade & $0.2034$ & $0.0322$ & $0.0148$ & $0.0073$ & $0.2083$ \\
    Second grade & $0.0346$ & $0.2023$ & $0.0571$ & $0.0096$ & $0.3169$ \\
    Third grade & $0.0175$ & $0.0490$ & $0.2371$ & $0.0255$ & $0.3420$ \\
    Lowest grade & $0.0058$ & $0.0120$ & $0.0258$ & $0.0658$ & $0.1328$ \\
    \hline
    Geometric margin & $0.2083$ & $0.3169$ & $0.3420$ & $0.1328$ & $1.0000$ \\
    \hline
  \end{tabular}
  \caption{Skew-symmetric independent PT orthogonal to Table \ref{tab:stuart_geometric marginal homogeneity}}
  \label{tab:stuart_geometric marginal heterogeneity}
  \centering
  \begin{tabular}{|c|cccc|c|}
    \hline
    \multirow{2}{*}{\diagbox{Right eye}{Left eye}} & Highest & Second & Third & Lowest & Geometric \\
     & grade & grade & grade & grade & margin\\
    \hline
    Highest grade & $0.0622$ & $0.0687$ & $0.0696$ & $0.0749$ & $0.2754$ \\
    Second grade & $0.0563$ & $0.0622$ & $0.0630$ & $0.0678$ & $0.2494$ \\
    Third grade & $0.0556$ & $0.0615$ & $0.0622$ & $0.0670$ & $0.2464$ \\
    Lowest grade & $0.0517$ & $0.0571$ & $0.0578$ & $0.0622$ & $0.2288$ \\
    \hline
    Geometric margin & $0.2259$ & $0.2495$ & $0.2526$ & $0.2720$ & $1.0000$ \\
    \hline
  \end{tabular}
\end{table}
Table \ref{tab:stuart_geometric marginal homogeneity} shows the closest geometric marginal homogeneous PT and Table \ref{tab:stuart_geometric marginal heterogeneity} shows the geometric marginal heterogeneous PT orthogonal to Table \ref{tab:stuart_geometric marginal homogeneity}.
The simplicial quasi-skewness $Q^2(\hat{\mbf{P}})$ is $0.080$ and the geometric marginal heterogeneity $M^2(\hat{\mbf{P}})$ is $0.139$.
These results reveal that the skewness in the vision data is primarily due to geometric marginal heterogeneity rather than quasi-skewness.
In practical terms, this analysis indicates that there are relative differences in the overall distribution of visual grades between the left and right eyes in this population of women in their factory.
However, the relatively small quasi-skewness indicates that the pattern of associations between specific vision grades for left and right eyes is fairly symmetric.

\begin{table}[t]
  \caption{Skewness array for Table \ref{tab:stuart_data} \citep{nakamura_symmetry_2024}}
  \label{table:skewness array}
  \centering
  \begin{tabular}{|r|r|r|r|r|}
    \hline
    \multirow{2}{*}{\diagbox{Right eye}{Left eye}} & Highest & Second & Third & Lowest \\
     & grade & grade & grade & grade \\
    \hline
    Highest grade & $0.00$ & $1.87$ & $0.38$ & $\mbf{41.80}$ \\
    Second grade & $-1.87$ & $0.00$ & $3.56$ & $-0.28$ \\
    Third grade & $-0.38$ & $-3.56$ & $0.00$ & $2.10$ \\
    Lowest grade & $\mbf{-41.80}$ & $0.28$ & $-2.10$ & $0.00$ \\
    \hline
  \end{tabular}
%
%
  \caption{Quasi-skewness array for Table \ref{tab:stuart_data}}
  \label{table:quasi-skewness array}
  \centering
  \begin{tabular}{|r|r|r|r|r|}
    \hline
    \multirow{2}{*}{\diagbox{Right eye}{Left eye}} & Highest & Second & Third & Lowest \\
     & grade & grade & grade & grade \\
    \hline
    Highest grade & $0.00$ & $-1.54$ & $-8.47$ & $\mbf{17.24}$ \\
    Second grade & $1.54$ & $0.00$ & $7.24$ & $-15.46$ \\
    Third grade & $8.47$ & $-7.24$ & $0.00$ & $-0.049$ \\
    Lowest grade & $\mbf{-17.24}$ & $15.46$ & $0.049$ & $0.00$ \\
    \hline
  \end{tabular}
%
%
  \caption{Geometric marginal heterogeneity array for Table \ref{tab:stuart_data}}
  \label{table:geometric marginal heterogeneity array}
  \centering
  \begin{tabular}{|r|r|r|r|r|}
    \hline
    \multirow{2}{*}{\diagbox{Right eye}{Left eye}} & Highest & Second & Third & Lowest \\
     & grade & grade & grade & grade \\
    \hline
    Highest grade & $0.00$ & $7.06$ & $8.90$ & $\mbf{24.67}$ \\
    Second grade & $-7.06$ & $0.00$ & $0.11$ & $5.34$ \\
    Third grade & $-8.90$ & $-0.11$ & $0.00$ & $3.93$ \\
    Lowest grade & $\mbf{-24.67}$ & $-5.34$ & $-3.93$ & $0.00$ \\
    \hline
  \end{tabular}
\end{table}

To further analyze the cell-wise contributions to skewness, quasi-skewness, and geometric marginal heterogeneity, we present the corresponding arrays in Tables \ref{table:skewness array}, \ref{table:quasi-skewness array}, and \ref{table:geometric marginal heterogeneity array}.
These arrays show the percentage contribution of each cell to the respective measure, with the sign indicating the direction of the effect.
In the skewness array (Table \ref{table:skewness array}) provided in \citet{nakamura_symmetry_2024}, we observe that the largest contributions to overall skewness come from the (Highest, Lowest) and (Lowest, Highest) cells, with values of $41.80\%$ and $-41.80\%$, respectively.
This indicates a strong skewness between these extreme categories.

The quasi-skewness array (Table \ref{table:quasi-skewness array}) reveals that the main contributors to quasi-skewness are also the (Highest, Lowest) and (Lowest, Highest) cells, but with somewhat smaller magnitudes ($17.24\%$ and $-17.24\%$).
Notably, the pattern of contributions across other cells differs substantially from the skewness array.
The second-largest contributions come from the (Second, Lowest) and (Lowest, Second) cells ($-15.46\%$ and $15.46\%$), followed by the (Highest, Third) and (Third, Highest) cells ($-8.47\%$ and $8.47\%$). 
These differences suggest that the cells contributing to skew-symmetry differ considerably when only the interaction component is considered.

Finally, the geometric marginal heterogeneity array (Table \ref{table:geometric marginal heterogeneity array}) shows that the (Highest, Lowest) and (Lowest, Highest) cells are the largest contributors ($24.67\%$ and $-24.67\%$).
This indicates that these cells also contribute substantially to the difference in geometric marginals between the left and right eyes.
Interestingly, this array exhibits yet another distinct pattern compared to the quasi-skewness array, with the second-largest contributions coming from the (Highest, Third) and (Third, Highest) cells ($8.90\%$ and $-8.90\%$), followed by (Highest, Second) and (Second, Highest) cells ($7.06\%$ and $-7.06\%$).
These differences highlight how cells contribute differently to geometric marginal heterogeneity than to skewness and quasi-skewness, with variations in both the magnitude and direction of contributions across different cell pairs.
These array analyses provide a detailed view of how each cell contributes to the overall asymmetry and its components, allowing a more comprehensive understanding of the structure of asymmetry in the data.

\section{Conclusions}\label{sec6}
This paper introduces a novel geometric approach to analyzing symmetric structures in square CTs using Aitchison geometry on the simplex.
We propose the concept of geometric marginal homogeneity, demonstrate how it relates to quasi-symmetry and complete symmetry, and prove that it forms a linear subspace in the simplex.
We develop simplicial measures for quasi-skewness and geometric marginal heterogeneity and prove an orthogonal decomposition of overall simplicial skewness.
This geometric perspective offers a coherent framework for analyzing both symmetry and skew-symmetry, allows for insightful orthogonal decompositions, and enables the development of new statistical measures.
An illustrative example demonstrated how this approach can complement traditional analyses.
By bridging compositional data analysis and CT analysis, this work opens up new possibilities for understanding and analyzing categorical data across various disciplines.


\section*{Acknowledgement}
This work was supported by the Research Institute for Mathematical Sciences, an International Joint Usage/Research Center located in Kyoto University.

\section*{Data Availability}
Not applicable.

\section*{Declarations}
\subsection*{Funding}
This work was supported by JST SPRING Grant Number JPMJSP2151 and JSPS KAKENHI Grant Number JP20K03756.

\subsection*{Conflict Interests}
The authors have no relevant financial or non-financial interests to disclose.

\begin{appendices}
\setcounter{theorem}{0}
\section{Proofs of Theorems}
\subsection{Proof of Theorem 1}
\begin{theorem}[Orthogonal decomposition]
  Let $\mbf{P}$ be an $I\times I\in\mcal{S}^{I^{2}}$ table. $\mbf{P}$ can be expressed in a unique way as
  \begin{align*}
    \mbf{P} = \mbf{P}_{\syind}\oplus\mbf{P}_{\skind}\oplus\mbf{P}_{\syint}\oplus\mbf{P}_{\skint},
  \end{align*}
  where $\mbf{P}_{\syind}\in\mcal{S}^{I^{2}}_{\syind}$, $\mbf{P}_{\skind}\in\mcal{S}^{I^{2}}_{\skind}$, $\mbf{P}_{\syint}\in\mcal{S}^{I^{2}}_{\syint}$, and $\mbf{P}_{\skint}\in\mcal{S}^{I^{2}}_{\skint}$.
\end{theorem}
\begin{proof}
  As stated in \citet{egozcue_independence_2015}, the initial orthogonal decomposition
  \[
    \mbf{P}=\mbf{P}_{\ind}\oplus\mbf{P}_{\inte}
  \]
  is uniquely derived through the orthogonal projection in $\mcal{S}_{\ind}^{I^{2}}$.
  Subsequently, based on the findings of \citet{nakamura_symmetry_2024}, the following two orthogonal decompositions,
  \begin{equation*}
    \mbf{P}_{\ind} = \mbf{P}_{\syind}\oplus\mbf{P}_{\skind}\quad \mathrm{and}\quad \mbf{P}_{\inte} = \mbf{P}_{\syint}\oplus\mbf{P}_{\skint},
  \end{equation*}
  are obtained in a unique way by each orthogonal projection onto $\mcal{S}_{\ind}^{I^{2}}$.
  Consequently, we obtain the unique orthogonal decomposition of $\mbf{P}$ into four components,
  \[
    \mbf{P} = \mbf{P}_{\syind}\oplus\mbf{P}_{\skind}\oplus\mbf{P}_{\syint}\oplus\mbf{P}_{\skint},   
  \]
  in a unique way.
\end{proof}
\subsection{Proof of Theorem 2}
\begin{theorem}[Subspace of quasi-symmetric PT]
  Quasi-symmetric PT given in Definition \ref{def:def2} is equivalent to the quasi-symmetric structure defined by \eqref{eq1}.
\end{theorem}
\begin{proof}
  Let $\mbf{P} = (p_{ij})$ be an $I \times I$ PT.
  First, we show that our definition implies Caussinus' definition.
  Assume $\mbf{P}$ satisfies Definition \ref{def:def2}, i.e., $\mbf{P}_{\inte} = T(\mbf{P}_{\inte})$.
  Let $\mbf{P}_{\ind} = row(\mbf{P}) \oplus col(\mbf{P})$ where $row(\mbf{P}])$ and $col(\mbf{P})$ are the geometric marginals.
  Then, $p_{ij} \propto g(row_i[\mbf{P}]) \cdot g(col_j[\mbf{P}]) \cdot c_{ij}$ where $c_{ij}$ are the entries of $\mbf{P}_{\inte}$.
  For any $i, j\in\{1,\ldots, I-1\}$, consider the odds ratio,
  \begin{align*}
    \theta_{ij} 
    &= \frac{p_{ij}p_{i+1,j+1}}{p_{i+1,j}p_{i,j+1}} \\
    &= \frac{g(row_i[\mbf{P}]) \cdot g(col_j[\mbf{P}]) \cdot c_{ij} \cdot g(row_{i+1}[\mbf{P}]) \cdot g(col_{j+1}[\mbf{P}]) \cdot c_{i+1,j+1}}{g(row_{i+1}[\mbf{P}]) \cdot g(col_j[\mbf{P}]) \cdot c_{i+1,j} \cdot g(row_i[\mbf{P}]) \cdot g(col_{j+1}[\mbf{P}]) \cdot c_{i,j+1}} \\
    &= \frac{c_{ij}c_{i+1,j+1}}{c_{i+1,j}c_{i,j+1}}.
  \end{align*}
  Since $\mbf{P}_{\inte} = T(\mbf{P}_{\inte})$, we have $c_{ij} = c_{ji}$ for $i, j\in\{1, \ldots, I\}$.
  Therefore, $\theta_{ij} = \theta_{ji}$.
  This is exactly Caussinus' definition of quasi-symmetry.

  Now, we demonstrate that Caussinus' definition implies our definition.
  Assume $\mbf{P}$ satisfies Caussinus' definition, i.e., $\theta_{ij} = \theta_{ji}$ for all $i, j\in\{1,\ldots, I-1\}$. We utilize the equivalent representation of the quasi-symmetry model,
  \[
    p_{ij} = \alpha_i \beta_j \psi_{ij}
  \]
  where $\alpha_i, \beta_j, \psi_{ij} > 0$ and $\psi_{ij} = \psi_{ji}$ for $i, j\in\{1, \ldots, I\}$.
  To prove that this implies our definition, we must decompose $\mbf{P}$ into $\mbf{P}_{\ind}$ and $\mbf{P}_{\inte}$ and demonstrate that $\mbf{P}_{\inte} = T(\mbf{P}_{\inte})$.
  We begin by calculating the geometric means of rows and columns,
  \[
    g(row_i[\mbf{P}]) = \left(\prod_{j=1}^I p_{ij}\right)^{1/I} = \alpha_i \left(\prod_{j=1}^I \beta_j \psi_{ij}\right)^{1/I},
  \]
  \[
    g(col_j[\mbf{P}]) = \left(\prod_{i=1}^I p_{ij}\right)^{1/I} = \beta_j \left(\prod_{i=1}^I \alpha_i \psi_{ij}\right)^{1/I}.
  \]
  We can now define the elements of $\mbf{P}_{\inte}$ as
  \[
    p_{ij}^{\inte}\propto\frac{p_{ij}}{g(row_i[\mbf{P}])g(col_j[\mbf{P}])} = \frac{\alpha_i \beta_j \psi_{ij}}{\alpha_i \left(\prod_{k=1}^I \beta_k \psi_{ik}\right)^{1/I}\beta_j \left(\prod_{k=1}^I \alpha_k \psi_{kj}\right)^{1/I}}.
  \]
  Note that the geometric means of $\alpha_i$ and $\beta_j$ over all $i$ and $j$ respectively are constants independent of any particular cell $(i,j)$.
  Therefore, we can simplify the expression for $p_{ij}^{\inte}$ as follows.
  \[
    p_{ij}^{\inte} \propto \frac{\psi_{ij}}{\left(\prod_{k=1}^I \psi_{ik}\right)^{1/I}\left(\prod_{k=1}^I \psi_{kj}\right)^{1/I}}.
  \]
  Similarly for $p_{ji}^{\inte}$,
  \[
    p_{ji}^{\inte} \propto \frac{\psi_{ji}}{\left(\prod_{k=1}^I \psi_{jk}\right)^{1/I}\left(\prod_{k=1}^I \psi_{ki}\right)^{1/I}}.
  \]
  Since $\psi_{ij} = \psi_{ji}$ for all $i, j\in\{1, \ldots, I\}$, we have
  \[
    p_{ij}^{\inte} = p_{ji}^{\inte}.
  \]
  This equality implies $\mbf{P}_{\inte} = T(\mbf{P}_{\inte})$.
  Thus, we have established the equivalence between our definition (Definition \ref{def:def2}) and Caussinus' definition.
\end{proof}
\subsection{Proof of Theorem 3}
\begin{theorem}[Subspace of quasi-symmetric PT]
  Let $\mcal{S}_{\qs}^{I^{2}}$ be the set of quasi-symmetric PTs.
  $\mcal{S}_{\qs}^{I^{2}}$ is a linear subspace of $\mcal{S}^{I^{2}}$.
  Its dimension is $(I-1)(I+4)/2$.
\end{theorem}
\begin{proof}
  If $\mbf{P}_{\inte}=T(\mbf{P}_{\inte})$, $\mbf{P}_{\skint} = \mbf{N}$, where $\mbf{N}$ is the additive identity that has the same entries.
  Thus, we can decompose $\mbf{P}$ into independent and symmetric-interaction parts,
  \[
    \mbf{P} = \mbf{P}_{\ind}\oplus\mbf{P}_{\syint},
  \]
  in a unique way.
  Therefore, $\mcal{S}_{\qs}^{I^{2}}$ is the direct sum of $\mcal{S}_{\ind}^{I^{2}}$ and $\mcal{S}_{\syint}^{I^{2}}$ and is indicated by
  \[
    \mcal{S}_{\qs}^{I^{2}} = \Set{\mbf{P}_{\ind}\oplus\mbf{P}_{\syint}\given\mbf{P}_{\ind}\in\mcal{S}_{\ind}^{I^{2}},\mbf{P}_{\syint}\in\mcal{S}_{\syint}^{I^{2}}}.
  \]
  For all $\mbf{P},\mbf{Q}\in\mcal{S}_{\qs}^{I^{2}}$, they can be decomposed into two parts,
  $\mbf{P}_{\ind},\mbf{Q}_{\ind}\in\mcal{S}_{\ind}^{I^{2}}$ and $\mbf{P}_{\syint}, \mbf{Q}_{\syint}\in\mcal{S}_{\syint}^{I^{2}}$.
  Both $\mcal{S}_{\ind}^{I^{2}}$ and $\mcal{S}_{\syint}^{I^{2}}$ are linear subspaces of $\mcal{S}^{I^{2}}$ and are closed under perturbation and powering.
  Thus, $\mcal{S}_{\qs}^{I^{2}}$ is also closed under them.
  The dimension of $\mcal{S}_{\qs}^{I^{2}}$ is the sum of the dimensions of $\mcal{S}_{\ind}^{I^{2}}$ and $\mcal{S}_{\syint}^{I^{2}}$.
  Therefore, the dimension of $\mcal{S}_{\qs}^{I^{2}}$ is $(2I-2)+I(I-1)/2 = (I-1)(I+4)/2$.
\end{proof}

\subsection{Proof of Theorem 4}
\begin{theorem}[Orthogonal projection onto quasi-symmetric PT]
Let $\mbf{P}\in\mcal{S}^{I^{2}}$ be an $I\times I$ table.
The closest quasi-symmetric PT to $\mbf{P}$ is $\mbf{P}_{\qs} = \mbf{P}_{\ind}\oplus\Big(0.5 \odot \big(\mbf{P}_{\inte} \oplus T(\mbf{P}_{\inte})\big)\Big)$ in the sense of the Aitchison distance in $\mcal{S}^{I^{2}}$.
\end{theorem}

\begin{proof}
  To prove this theorem, for any $\mbf{Q}\in\mcal{S}_{\qs}^{I^2}$, we must show $\langle\mbf{P}\ominus\mbf{P}_{\qs}, \mbf{P}_{\qs}\ominus\mbf{Q}\rangle_{A} = 0$.
  The $(i,j)$-th entry of $\clr(\mbf{P}\ominus\mbf{P}_{\qs})$ is given as
  \begin{align*}
    \clr_{ij}(\mbf{P}\ominus\mbf{P}_{\qs}) 
    &= \clr_{ij}(\mbf{P}) - \clr_{ij}(\mbf{P}_{\qs}) \\
    &= \clr_{ij}(\mbf{P}_{\inte}) - 0.5\big(\clr_{ij}(\mbf{P}_{\inte}) + \clr_{ji}(\mbf{P}_{\inte})\big).
  \end{align*}
  Hence, $\mbf{P}\ominus\mbf{P}_{\qs}$ is, by construction, a pure interaction component.
  Therefore it lies in the interaction subspace $\mcal{S}_{\inte}^{I^{2}}$.
  Moreover, for any $i,j\in\{1,\ldots,I\}$, we have
  \[
  \clr_{ji}\!\left(\mbf{P}\ominus\mbf{P}_{\qs}\right)=-\,\clr_{ij}\!\left(\mbf{P}\ominus\mbf{P}_{\qs}\right),
  \]
  indicating the table is skew-symmetric and thus belongs to the subspace $\mcal{S}_{\skw}^{I^{2}}$ as well.
  In other words, the table lies in $\mcal{S}_{\inte}^{I^{2}}\cap\mcal{S}_{\skw}^{I^{2}} =\mcal{S}_{\skint}^{I^{2}}$.  
  
  Since \(\mbf{Q}\in\mcal{S}_{\qs}^{I^{2}}\), its interaction component is symmetric, that is, for any $i,j\in\{1,\ldots,I\}$,
  \[
    \clr_{ij}(\mbf{Q}_{\inte})=\clr_{ji}(\mbf{Q}_{\inte}).
  \]
  Hence
  \[
  \clr_{ij}(\mbf{Q}_{\inte}) = 0.5\big(\clr_{ij}(\mbf{Q}_{\inte})+\clr_{ji}(\mbf{Q}_{\inte})\big).
  \]
  Consequently, separating the independent part from the interaction part, we can rewrite the difference as
  \begin{align*}
    \clr_{ij}(\mbf{P}_{\qs}\ominus\mbf{Q}) = \clr_{ij}(\mbf{P}_{\ind}\ominus\mbf{Q}_{\ind}) + 0.5\big(\clr_{ij}(\mbf{P}_{\inte}\ominus\mbf{Q}_{\inte}) + \clr_{ji}(\mbf{P}_{\inte}\ominus\mbf{Q}_{\inte})\big).
  \end{align*}

  The first term is an element of the independent subspace $\mcal{S}_{\ind}^{I^{2}}$, whereas the second term lies simultaneously in the interaction and symmetric subspaces, that is, in $\mcal{S}_{\inte}^{I^{2}}\cap\mcal{S}_{\sym}^{I^{2}} =\mcal{S}_{\syint}^{I^{2}}$.
  
  Then we have the Aitchison inner product of $\mbf{P}\ominus\mbf{P}_{\qs}$ and $\mbf{P}_{\qs}\ominus\mbf{Q}$ as follows.
  \begin{align*}
    &\langle\mbf{P}\ominus\mbf{P}_{\qs}, \mbf{P}_{\qs}\ominus\mbf{Q}\rangle_{A} \\
    &\quad= \sum_{i=1}^{I}\sum_{j=1}^{I}\clr_{ij}(\mbf{P}\ominus\mbf{P}_{\qs})\clr_{ij}(\mbf{P}_{\qs}\ominus\mbf{Q}) \\
    &\quad= \sum_{i=1}^{I}\sum_{j=1}^{I}\clr_{ij}(\mbf{P}\ominus\mbf{P}_{\qs})\clr_{ij}(\mbf{P}_{\ind}\ominus\mbf{Q}_{\ind}) \\
    &\qquad+ \sum_{i=1}^{I}\sum_{j=1}^{I}\clr_{ij}(\mbf{P}\ominus\mbf{P}_{\qs})\Big(0.5\big(\clr_{ij}(\mbf{P}_{\inte}\ominus\mbf{Q}_{\inte}) + \clr_{ji}(\mbf{P}_{\inte}\ominus\mbf{Q}_{\inte})\big)\Big).
  \end{align*}
  Here, the first inner product is zero because the interaction subspace $\mcal{S}_{\inte}^{I^{2}}$ is orthogonal to the independent subspace \(\mcal{S}_{\ind}^{I^{2}}\).
  The second inner product is zero because, inside the interaction subspace, the skew-symmetric component $\mcal{S}_{\skint}^{I^{2}}$ is orthogonal to the symmetric component $\mcal{S}_{\syint}^{I^{2}}$.
  Hence
  \[
    \langle\mbf{P}\ominus\mbf{P}_{\qs}, \mbf{P}_{\qs}\ominus\mbf{Q}\rangle_{A} = 0,
  \]
  which completes the proof.
\end{proof}

\subsection{Proof of Theorem 5}
\begin{theorem}[Subspace of geometric marginal homogeneity PT]
  Let $\mcal{S}_{\gmh}^{I^{2}}$ be the set of geometric marginal homogeneous PTs.
  $\mcal{S}_{\gmh}^{I^{2}}$ is a linear subspace of $\mcal{S}^{I^{2}}$. Its dimension is $I(I-1)$.
\end{theorem}
\begin{proof}
  If $\mbf{P}_{\ind}=T(\mbf{P}_{\ind})$, $\mbf{P}_{\skind} = \mbf{N}$, where $\mbf{N}$ is the additive identity that has the same entries.
  Thus, we can decompose $\mbf{P}$ into independent symmetric and interaction parts,
  \[
    \mbf{P} = \mbf{P}_{\syind}\oplus\mbf{P}_{\inte},
  \]
  in a unique way.
  Therefore, $\mcal{S}_{\gmh}^{I^{2}}$ is the direct sum of $\mcal{S}_{\syind}^{I^{2}}$ and $\mcal{S}_{\inte}^{I^{2}}$ and indicated by
  \[
    \mcal{S}_{\gmh}^{I^{2}} = \Set{\mbf{P}_{\syind}\oplus\mbf{P}_{\inte}\given\mbf{P}_{\syind}\in\mcal{S}_{\syind}^{I^{2}},\mbf{P}_{\inte}\in\mcal{S}_{\inte}^{I^{2}}}.
  \]
  For all $\mbf{P},\mbf{Q}\in\mcal{S}_{\gmh}^{I^{2}}$, they can be decomposed into two parts, $\mbf{P}_{\syind},\mbf{Q}_{\syind}\in\mcal{S}_{\syind}^{I^{2}}$ and $\mbf{P}_{\inte}, \mbf{Q}_{\inte}\in\mcal{S}_{\inte}^{I^{2}}$.
  Both $\mcal{S}_{\syind}^{I^{2}}$ and $\mcal{S}_{\inte}^{I^{2}}$ are linear subspaces of $\mcal{S}^{I^{2}}$ and are closed under perturbation and powering.
  Thus, $\mcal{S}_{\gmh}^{I^{2}}$ is also closed under them.
  The dimension of $\mcal{S}_{\gmh}^{I^{2}}$ is the sum of the dimensions of $\mcal{S}_{\syind}^{I^{2}}$ (which is $I-1$) and $\mcal{S}_{\inte}^{I^{2}}$ (which is $I^2-2I+1$).
  Therefore, the dimension of $\mcal{S}_{\gmh}^{I^{2}}$ is $(I-1) + (I^2-2I+1) = I(I-1)$.
\end{proof}

\subsection{Proof of Theorem 6}
\begin{theorem}[Orthogonal projection onto geometric marginal homogeneous PT]
Let $\mbf{P}\in\mcal{S}^{I^{2}}$ be an $I\times I$ table. The closest geometric marginal homogeneous PT to $\mbf{P}$ is $\mbf{P}_{\gmh} = \Big(0.5 \odot \big( \mbf{P}_{\ind} \oplus T(\mbf{P}_{\ind}) \big) \Big) \oplus \mbf{P}_{\inte}$ in the sense of the Aitchison distance in $\mcal{S}^{I^{2}}$.
\end{theorem}

\begin{proof}
  To prove this theorem, for any $\mbf{Q}\in\mcal{S}_{\gmh}^{I^2}$, we must show $\langle\mbf{P}\ominus\mbf{P}_{\gmh}, \mbf{P}_{\gmh}\ominus\mbf{Q}\rangle_{A} = 0$.
  The $(i,j)$-th entry of $\clr(\mbf{P}\ominus\mbf{P}_{\gmh})$ is given as
  \begin{align*}
    \clr_{ij}(\mbf{P}\ominus\mbf{P}_{\gmh}) 
    &= \clr_{ij}(\mbf{P}) - \clr_{ij}(\mbf{P}_{\gmh}) \\
    &= \clr_{ij}(\mbf{P}_{\ind}) - 0.5\big(\clr_{ij}(\mbf{P}_{\ind}) + \clr_{ji}(\mbf{P}_{\ind})\big).
  \end{align*}
  Hence, $\mbf{P}\ominus\mbf{P}_{\gmh}$ is, by construction, a pure independent component.
  Therefore it lies in the independent subspace $\mcal{S}_{\ind}^{I^{2}}$.
  Moreover, for any $i,j\in\{1,\ldots,I\}$, we have
  \[
    \clr_{ji}\!\left(\mbf{P}\ominus\mbf{P}_{\gmh}\right)=-\,\clr_{ij}\!\left(\mbf{P}\ominus\mbf{P}_{\gmh}\right),
  \]
  indicating the table is skew-symmetric and thus belongs to the subspace $\mcal{S}_{\skw}^{I^{2}}$ as well.
  In other words, the table lies in $\mcal{S}_{\ind}^{I^{2}}\cap\mcal{S}_{\skw}^{I^{2}} =\mcal{S}_{\skind}^{I^{2}}$.  
  
  Since \(\mbf{Q}\in\mcal{S}_{\gmh}^{I^{2}}\), its independent component is symmetric, so for all $i,j\in\{1,\ldots,I\}$,
  \[
    \clr_{ij}(\mbf{Q}_{\ind})=\clr_{ji}(\mbf{Q}_{\ind}).
  \]
  Hence
  \[
    \clr_{ij}(\mbf{Q}_{\ind}) = 0.5\big(\clr_{ij}(\mbf{Q}_{\ind})+\clr_{ji}(\mbf{Q}_{\ind})\big).
  \]
  Consequently, separating the interaction part from the independent part, we can rewrite the difference as
  \begin{align*}
    \clr_{ij}(\mbf{P}_{\gmh}\ominus\mbf{Q}) = \clr_{ij}(\mbf{P}_{\inte}\ominus\mbf{Q}_{\inte}) + 0.5\big(\clr_{ij}(\mbf{P}_{\ind}\ominus\mbf{Q}_{\ind}) + \clr_{ji}(\mbf{P}_{\ind}\ominus\mbf{Q}_{\ind})\big).
  \end{align*}

  The first term is an element of the interaction subspace $\mcal{S}_{\inte}^{I^{2}}$, whereas the second term lies simultaneously in the independent and symmetric subspaces, that is, in $\mcal{S}_{\ind}^{I^{2}}\cap\mcal{S}_{\sym}^{I^{2}} =\mcal{S}_{\syind}^{I^{2}}$.
  
  Then we have the Aitchison inner product of $\mbf{P}\ominus\mbf{P}_{\gmh}$ and $\mbf{P}_{\gmh}\ominus\mbf{Q}$ as follows.
  \begin{align*}
    &\langle\mbf{P}\ominus\mbf{P}_{\gmh}, \mbf{P}_{\gmh}\ominus\mbf{Q}\rangle_{A} \\
    &\quad= \sum_{i=1}^{I}\sum_{j=1}^{I}\clr_{ij}(\mbf{P}\ominus\mbf{P}_{\gmh})\clr_{ij}(\mbf{P}_{\gmh}\ominus\mbf{Q}) \\
    &\quad= \sum_{i=1}^{I}\sum_{j=1}^{I}\clr_{ij}(\mbf{P}\ominus\mbf{P}_{\gmh})\clr_{ij}(\mbf{P}_{\inte}\ominus\mbf{Q}_{\inte}) \\
    &\qquad+ \sum_{i=1}^{I}\sum_{j=1}^{I}\clr_{ij}(\mbf{P}\ominus\mbf{P}_{\gmh})\Big(0.5\big(\clr_{ij}(\mbf{P}_{\ind}\ominus\mbf{Q}_{\ind}) + \clr_{ji}(\mbf{P}_{\ind}\ominus\mbf{Q}_{\ind})\big)\Big).
  \end{align*}
  Here, the first inner product is zero because the independent subspace $\mcal{S}_{\ind}^{I^{2}}$ is orthogonal to the interaction subspace \(\mcal{S}_{\inte}^{I^{2}}\).
  The second inner product is zero because, inside the independent subspace, the skew-symmetric component $\mcal{S}_{\skind}^{I^{2}}$ is orthogonal to the symmetric component $\mcal{S}_{\syind}^{I^{2}}$.
  Hence
  \[
    \langle\mbf{P}\ominus\mbf{P}_{\gmh}, \mbf{P}_{\gmh}\ominus\mbf{Q}\rangle_{A} = 0,
  \]
  which completes the proof.
\end{proof}

\subsection{Proof of Theorem 7}
\begin{theorem}[Orthogonal decomposition of simplicial skewness]
  Given a PT $\mbf{P}$, its simplicial skewness $E^{2}(\mbf{P})$ can be decomposed into simplicial quasi-skewness $Q^{2}(\mbf{P})$ and geometric marginal heterogeneity $M^{2}(\mbf{P})$, that is,
  \begin{align*}
    E^{2}(\mbf{P}) = Q^{2}(\mbf{P}) + M^{2}(\mbf{P}).
  \end{align*}
\end{theorem}

\begin{proof}
For any $\mbf{P} \in \mcal{S}^{I^2}$, we can write its unique orthogonal decomposition
\[
  \mbf{P} = \mbf{P}_{\syind} \oplus \mbf{P}_{\skind} \oplus \mbf{P}_{\syint} \oplus \mbf{P}_{\skint}.
\]
The skew-symmetric component of $\mbf{P}$ is given by
\[
  \mbf{P}_{\skw} = \mbf{P}_{\skind} \oplus \mbf{P}_{\skint}.
\]
By definition, the simplicial skewness is
\[
  E^2(\mbf{P}) = \|\mbf{P}_{\skw}\|_{A}^2 = \|\mbf{P}_{\skind} \oplus \mbf{P}_{\skint}\|_{A}^2.
\]
Since $\mbf{P}_{\skind}$ and $\mbf{P}_{\skint}$ are orthogonal in Aitchison geometry, we have
\[
  \|\mbf{P}_{\skw}\|_{A}^2 = \|\mbf{P}_{\skind}\|_{A}^2 + \|\mbf{P}_{\skint}\|_{A}^2.
\]
This directly implies
\[
  E^2(\mbf{P}) = Q^2(\mbf{P}) + M^2(\mbf{P}).
\]
Thus, we have shown that the simplicial skewness can be orthogonally decomposed into simplicial quasi-skewness and geometric marginal heterogeneity.
\end{proof}

\end{appendices}

\bibliography{references}

\end{document}